\documentclass[11pt]{article}
\usepackage{a4wide}
\usepackage{amsmath}

\usepackage{amsfonts}
\usepackage{amssymb}
\usepackage{euscript}
\newenvironment{proof}{\noindent {\it Proof.~~}\ }{\  \rule{1mm}{2mm}\medskip}

\newenvironment{proof*}{\noindent {\it Proof.~~}\ }{}

\newenvironment{proofof}[2]{\noindent {\it Proof of #1}~#2: \
}{~\rule{1mm}{2mm}\medskip}
\newtheorem{theorem}{Theorem}
\newtheorem{lemma}[theorem]{Lemma}
\newtheorem{corollary}[theorem]{Corollary}

\def\Z{\mathbb Z}

\newtheorem{theirtheorem}{Theorem}

\newtheorem{theirlemma}[theirtheorem]{Lemma}

\usepackage{latexsym,amssymb}
\usepackage{amsmath}

\def\Z{\mathbb Z}

\begin{document}
\title{ Distinct Matroid Base Weights  and Additive Theory
 }
\author{ {Y. O. Hamidoune}\thanks{
Universit\'e Pierre et Marie Curie, E. Combinatoire, Case 189, 4
Place Jussieu, 75005 Paris, France. \texttt{yha@ccr.jussieu.fr}}
\and {I.P. da Silva}\thanks{
CELC/Universidade de Lisboa, Faculdade de Ci\^encias, Campo Grande, edif\'icio C6 - Piso 2,
1749-016 Lisboa, Portugal.\texttt{isilva@cii.fc.ul.pt}
} }

\date{}

\maketitle

\begin{abstract}
Let $M$ be a matroid on a set $E$ and let  $w:E\longrightarrow G$ be a weight function, where $G$ is a cyclic group. Assuming that $w(E)$ satisfies the Pollard's Condition (i.e. Every non-zero element  of $w(E)-w(E)$ generates $G$),
we obtain a formulae for the number of distinct base weights. If $|G|$ is a prime, our result coincides with a result Schrijver and Seymour.

We also describe Equality cases in this formulae. In the prime case,
our result generalizes Vosper's Theorem.
\end{abstract}

\section{Introduction}


Let $G$ be a finite cyclic group and let $A,B$ be nonempty subsets of $G.$
The starting point of Minkowski set sum estimation is the inequality
$|A+B|\ge \min(|G|,|A|+|B|-1),$ where
$|G|$ is a prime, proved by Cauchy \cite{CAU} and rediscovered
by Davenport \cite{DAV}. The first generalization of this  result,
 due to Chowla \cite{CHO}, states that $|A+B|\ge \min(|G|,|A|+|B|-1),$ if
 there is a $b\in B$ such that  every non-zero element  of $B-b$ generates $G.$ In order to generalize his extension of the Cauchy-Davenport Theorem \cite{P1} to composite moduli,
Pollard introduced  in \cite{P2} the following more sophisticated Chowla type condition:
Every non-zero element  of $B-B$ generates $G.$

Equality cases of the Cauchy-Davenport were determined by Vosper in \cite{vosper1,vosper2}.
Vosper's Theorem was generalized by Kemperman \cite{KEM}. We need only a light form of Kemperman's result
stated in the beginning of Kemperman's paper.

We need the following combination of Chowla and Kemperman results:

\begin{theirtheorem}(Chowla \cite{CHO}, Kemperman \cite{KEM})\label{vosper}
Let  $A,B$ be non-empty subsets of a cyclic group $G$ with
 $|A|,|B|\ge 2$ such that for some $b\in B,$
 every non-zero element  of $B-b$ generates $G.$ Then   $|A+B|\geq |A|+|B|-1.$

Moreover  $|A+B|= |A|+|B|-1$  if and only if  $A+B$  is an arithmetic progression.
\end{theirtheorem}

A shortly proved generalization of this result to non-abelian groups is obtained in \cite{halgebra}.

Zero-sum problems form another developing area in Additive Combinatorics having several applications.
 The   Erd{\H o}s-Ginzburg-Ziv Theorem \cite{egz} was the starting point of this area. This result states that
a sequence of elements of an abelian group $G$ with length $\ge 2|G|-1$ contains a zero-sum subsequence of length $=|G|$.

The reader may find some details on these two areas of Additive Combinatorics in the
text books:  Nathanson \cite{NAT},  Geroldinger-Halter-Koch \cite{geroldinger} and Tao-Vu \cite{tv}.
More specific questions may be found in Caro's survey paper \cite{caro}.

The notion of a matroid was introduced by Whitney in 1935 as a generalization of a matrix.
Two pioneer works connecting matroids and Additive Combinatorics are due to
Schrijver-Seymour \cite{SS}, Dias da Silva-Nathanson \cite{DN}. Recently, in  \cite{IS}, orientability of matroids is naturally related with an open problem on Bernoulli matrices.

Stating the first  result requires some vocabulary:

Let $E$ be a finite set. The set of the subsets of $E$ will be denoted by $2^E.$

A {\it matroid} over   $E$  is an ordered pair  $(E,\mathcal B)$  where  $\mathcal B\subseteq 2^E$    satisfies the following axioms:

\begin{itemize}
  \item[(B1)]  $\mathcal B\not=\emptyset$.
  \item[(B2)] For all   $B,B'\in \mathcal B,$  if  $B\subseteq B'$  then  $B=B'$.
  \item[(B3)] For all   $B,B'\in \mathcal B$  and  $x\in B\setminus B',$  there is a $y\in B'\setminus B$  such that  $(B\setminus \{x\}) \cup \{y\}\in \mathcal B$.

\end{itemize}

A set  belonging to $\mathcal B$  is called a {\it basis } of the matroid $M.$

The {\it rank } of  a subset  $A\subseteq E$  is by definition   $r_M(A):= max \{|B\cap A|: B \ \mbox{is a basis  of} \ M\}$. We write $r(M)=r(E).$ The reference to $M$ could be omitted.  A {\it hyperplane} of the matroid $M$ is a maximal subset of $E$ with rank  $=r(M)-1$.

 The {\it uniform} matroid of rank $r$ on a set $E$ is by definition ${\cal U}_r(E)=(E, {E\choose r}),$
where  ${E\choose r}$ is the set of all $r$-subsets of $E.$  Let $M$ be a matroid on $E$ and let $N$ be a matroid on $F.$ We define the direct sum:

$$M\oplus N=(E\times \{0\}\cup F\times \{1\}, \{ B\times \{0\} \cup C\times \{1\}: B \ \mbox{ is a base of}\  M
\  \mbox{and} \ C \ \mbox{ is a base of} \ N\}.$$

Let    $w:E\longrightarrow G$ be a weight function, where   $G$  is an abelian group. The weight of a subset $X$
is by definition $$X^w=\sum _{x\in X} w(x).$$
The set of distinct base weights is
$$M^w=\{B^w : B\ is\ a\ basis\ of\ M\}.$$

Suppose now $|G|=p$ is a prime number.
Schrijver and Seymour proved that $|M^w|\ge \min (p,\sum _{g\in G} r (w^{-1}(g))-r (M)+1).$ Let $A$ and $B$ be subsets of $G.$ Define  $w: A\times \{0\}\cup B\times \{1\},$
by the relation $w(x,y)=x.$
Then $$({\cal U}_1(A)\oplus {\cal U}_1(B))^w=A+B.$$
Applying their result to this matroid, Schrijver and Seymour obtained the  Cauchy-Davenport Theorem.

Let $x_1, \ldots, x_{2p-1}\in G.$ Consider the uniform matroid  $M={\cal U}_p(E),$  of rank  $p$  over the set
$E=\{1, \ldots ,2p-1\},$ with weight function  $w(i)=x_i.$
In order to prove the Erd{\H o}s-Ginzburg-Ziv Theorem \cite{egz}, one may clearly assume that no element is repeated $p$ times. In particular for every $g\in G,$ $r(w^{-1}(g))=|w^{-1}(g)|.$ Applying  Schrijver and Seymour to this matroid we have:
$$|M^w|\ge \min (|G|,\sum _{g\in G} r (w^{-1}(g))-r (M)+1)= \min (p,\sum _{g\in G} |w^{-1}(g)|-p+1)=p.$$
Thus Schrijver-Seymour result also implies the Erd{\H o}s-Ginzburg-Ziv Theorem \cite{egz} in a prime order.

In the present work, we prove the following result:

\begin{theorem} \label{main}
Let $G$ be a cyclic group,  $M$  be a  matroid on a finite set $E$ with $r(M)\ge 1$ and let  $w:E\longrightarrow G$ be a weight function. Assume moreover that every non-zero element  of $w(E)-w(E)$ generates $G.$
Then
\begin{equation}\label{ss}
|M^w|\ge \min (|G|,\sum _{g\in G} r (w^{-1}(g))-r (M)+1),
\end{equation} where $M^w$ denotes the set of distinct base weights.
Moreover, if  Equality holds in (\ref{ss}) then
one of the following conditions holds:

\begin{itemize}
  \item[(i)] $r(M)=1$ or  $M^w$ is an  arithmetic progression.
  \item[(ii)]  There is a hyperplane  $H$  of  $M$  such
  that $M^w= g+(M/H)^w,$  for some $g\in G.$
\end{itemize}

\end{theorem}

If $G$ has a prime order, then the condition on $w(E)-w(E)$ holds trivially. In this case (\ref{ss})
reduces to the result of Schrijver-Seymour.

\section{Terminology and Preliminaries}

Let $M$ be a matroid on a finite set $E$.
 One may see easily from the definitions that all bases a matroid have the same cardinality.
 A {\it circuit } of $M$ is a minimal set not contained in a base. A loop is an element $x$ such that $\{x\}$
 is a circuit. By the definition  bases contain no  loop.
The closure of a subset  $A\subseteq E$  is by definition  $$cl(A)=\{ x\in A:\ r(A\cup x)=r(A)\}.$$ Note that an element  $x\in cl(A)$  if and only if   $x\in A,$  or   there is circuit  $C$  such   $x\in C$  and  $C\setminus \{x\}\subseteq A$.


Given a matroid  $M$  on a set $E$ and a subset  $A\subseteq E.$ Then
$\mathcal B/A:=\{J\setminus A: J \mbox{ is a  basis of }\  M \ \mbox{with}\  |B\cap A|=r(A)  \}$. One may see easily that $M/A=(E\setminus A, \mathcal B/A)$ is a matroid on $E\setminus A.$ We say that this  matroid is obtained from  $M$  contracting  $A.$ Notice that $r_{M/A}(X)=r_M(X\cup A)-r_M(A).$


Recall the following easy lemma:

\begin{lemma}\label{contraction}

Let  $M$  be a  matroid on a finite set $E$ and let $U,V$ be disjoint subsets of $E.$ Then

 \begin{itemize}
   \item $M/U$ and $M/cl(U)$ have the same bases. In particular, $(M/U)^w=(M/cl(U))^w.$
   \item  $(M/U)/V=M/(U\cup V).$
 \end{itemize}
\end{lemma}

For more details on matroids, the reader may refer to one of the text books: Welsh \cite{welsh} or White \cite{white}.

For   $u\in E,$  we put
$$G_{u}:=\{g\in G:\  u\in cl(w^{-1}(g))\}.$$

We recall the following lemma proved by  Schrijver and Seymour in \cite{SS}:

\begin{theirlemma}\label{sslem}

Let  $M$  be a matroid on a finite set $E$ and let  $w:E\longrightarrow G$ be a weight function. Then for every
non-loop element $u\in E,$
$$ (M/u)^w +G_u \subseteq M^w.$$

\end{theirlemma}

\begin{proof} Take a basis  $B$  of  $M/u$  and an element  $g\in G_u$. If  $g=w(u)$  then, by definition of contraction,
 $B\cup \{u\}$  is a basis of  $M$ and  $B^w+w(u)\in M^w$. If  $g\neq w(u),$  there is a circuit  $C$  containing  $u$  such that
$\emptyset\not= C\setminus \{u\}\subseteq w^{-1}(g)$. For some  $v\in C\setminus  \{u\}$  the subset  $B\cup  \{v\}$  must be a basis of  $M$  otherwise  $C\setminus  \{v\}\subseteq cl(B)$, implying that  $u\in cl(B),$ in contradiction with the assumption that  $B$  is a basis of  $M/u$. Therefore  $(B\cup  \{v\})^w=B^w+g\in M^w$.\end{proof}

\section{Proof of the main result}

We shall now prove our result:

\begin{proofof}{Theorem}{\ref{main}}

 We first prove (\ref{ss}) by induction on the rank of  $M$. The result holds trivially if $r(M)=1.$
 Since $r(M)\ge 1,$ $M$ contains a non-loop element.
 Take an arbitrary non-loop element $y.$

\begin{eqnarray}
|M^w|&\ge&|(M/y)^w+G_y |\nonumber \\&\ge& |(M/y)^w|+|G_y|-1\nonumber \\&\ge& \sum _{g\in G} r (w^{-1}(g))-r (M)+1.\label{pmain}
\end{eqnarray}

The first inequality follows from Lemma \ref{sslem}, the second follows by  Theorem \ref{vosper} and the third is a direct consequence of the definitions of  $M/u$  and  $G_u$. This proves the first part of the theorem.

Suppose now that Equality holds in (\ref{ss}) and that  Condition (i)  is not satisfied.
In particular $r(M)\ge 2.$ Also $|M^w|\geq 2,$ otherwise $M^w$ is a progression, a contradiction.

We claim that there exits a non-loop element  $u\in E$  such that  $|(M/u)^w|\geq 2.$  Assume on the contrary that for every non-loop element  $u\in E$  we have  $|(M/u)^w|= 1$. Then every pair of bases  $B_1,B_2$  of  $M$  with  $B_1^w\not=B_2^w$. satisfies  $B_1\cap B_2=\emptyset$ otherwise for every  $z\in B_1\cap B_2$, $|(M/z)^w|\geq 2$. Now, for  every  $z\in B_1,$ there is $z'\in B_2$ such that $C=(B_1\setminus \{z\})\cup \{z'\}$ is a base of M. For such a base  $C$,  $B_1\cap C\neq \emptyset,$ $B_2\cap C\neq \emptyset,$ and we must have $B_1^w =C^w= B_2^w ,$ a contradiction.

Applying the chain of inequalities proving (\ref{pmain})  with  $y=u$. We have

\begin{equation}\label{V}
|M^w|= |(M/u)^w+G_u|=|(M/u)^w|+|G_u|-1.
\end{equation}

Note that $w(E\setminus \{u\})\subset w(E)$, clearly verifies  the Pollard condition. If  $|G_u|\geq 2$ Theorem \ref{vosper} implies that  $M^w$ is a progression and thus  $M$  satisfies Condition (i) of the theorem, contradicting our assumption on  $M$. We must have  $|G_u|=1.$

Thus  $G_u=\{w(u)\}$ and  $M^w=w(u)+(M/u)^w.$

Since the translate of a progression is a progression, $M/u$ is not a progression.
By Lemma \ref{contraction}, $(M/u)$ and $M/cl(u)$ have the same bases and thus the result holds if
$r(M)=2.$ If $r(M)>2,$ then by the Induction hypothesis there is a hyperplane $H$ of $M/u$ such that
$(M/u)^w=(M/u/H)^w=(M/(Cl(\{u\}\cup H))^w,$  and (ii) holds.\end{proofof}

\begin{corollary}\label{vospers}(Vosper's Theorem \cite{vosper1,vosper2})
Let $p$ be a prime and let $A,B$ be subsets of $\Z_p$ such that
 $|A|,|B|\ge 2.$

If  $|A+B|= |A|+|B|-1<p$ then one of the following holds:
\begin{itemize}
  \item[(i)] $c-A= (\Z_p\setminus B)$.
  \item[(ii)] $A$ and $B$ are arithmetic progressions with a same difference.
\end{itemize}

\end{corollary}
\begin{proof} Consider the matroid $N=({\cal U}_1(A)\oplus {\cal U}_1(B))$  and its weight function  $w$ defined in the Introduction.  $H=A\times \{0\}$  and $H'=B\times \{1\}$ are the hyperplanes of $N$ and we have  $N^w=A+B$.

If  $|N^w|=|A|+|B|-1$  then Theorem \ref{main}  says that   $N$  must satisfy one of its conditions  (i) or (ii). Since by hypothesis  $|A|,|B|\geq 2$  we have  $|N^w|> \max (|A|,|B|) \ge |(N/H)^w|,|(N/H')^w|$  and we conclude that  $N^w$  must be an arithmetic progression with difference $d$. Without loss of generality we may take $d=1.$

{\bf Case} 1. $|A+B|=p-1$. Put $\{c\}=\Z_p\setminus (A+B)$.
We have $c-A\subset (\Z_p\setminus B)$. Since these sets have the same cardinality
we have $c-A= (\Z_p\setminus B)$.

{\bf Case} 2. $|A+B|<p-1$.

We have $|A+B+\{0,1\}|=|A+B|+1=|A|+|B|<p.$

We must have  $|A+\{0,1\}|=|A|+1$, since otherwise
by the Cauchy-Davenport Theorem,

\begin{eqnarray*}
|A+B|+1&=&|A+B+\{0,1\}|\\&=& |A+\{0,1\}+B|\\
&\ge&  (|A|+2)+|B|-1 =|A|+|B|+1,
\end{eqnarray*}
a contradiction.
It follows that $A$ is an arithmetic progression with difference $1$.
Similarly $B$ is an arithmetic progression with difference $1$.\end{proof}

\end{document}